\documentclass{elsarticle}
\usepackage{enumerate}
\usepackage{amssymb}
\usepackage{amsthm}
\usepackage{amsmath}

\newtheorem{thm}{Theorem}[section]
\newtheorem{lem}[thm]{Lemma}
\newtheorem{cor}[thm]{Corollary}
\newtheorem{defi}[thm]{Definition}

\begin{document}

\title{Dixmier traces are weak$^*$ dense in the set of all fully symmetric traces}

\author{F. Sukochev}
\ead{f.sukochev@unsw.edu.au}
\address{School of Mathematics and Statistics, University of New South Wales, Sydney, 2052, Australia.}
\author{D. Zanin\corref{cor1}}
\ead{d.zanin@unsw.edu.au}
\address{School of Mathematics and Statistics, University of New South Wales, Sydney, 2052, Australia.}
\cortext[cor1]{Corresponding author}

\begin{abstract} We extend Dixmier's construction of singular traces (see \cite{Dixmier}) to arbitrary fully symmetric operator ideals. In fact, we show that the set of Dixmier traces is weak$^*$ dense in the set of all fully symmetric traces (that is, those traces which respect Hardy-Littlewood submajorization). Our results complement and extend earlier work of Wodzicki \cite{Wodzicki}.
\end{abstract}


\maketitle

\section{Introduction}

In his groundbreaking paper \cite{Dixmier}, J. Dixmier proved the existence of singular traces (that is, linear positive unitarily invariant functionals which vanish on all finite dimensional operators) on the algebra $B(H)$ of all bounded linear operators acting on infinite-dimensional separable Hilbert space $H.$ Namely, if $\psi:\mathbb{R}_+\to\mathbb{R}_+$ is a concave increasing function such that
\begin{equation}\label{simple psi cond}
\lim_{t\to\infty}\frac{\psi(2t)}{\psi(t)}=1,
\end{equation}
then there is a singular trace $\tau_{\omega},$ defined for every positive compact operator $A\in B(H)$ by setting
\begin{equation}\label{dixmier trace}
\tau_{\omega}(A)=\omega(\frac1{\psi(n+1)}\sum_{k=0}^n\mu(k,A)).
\end{equation}
Here, $\omega$ is an arbitrary dilation invariant singular state on the algebra $l_{\infty}$ of all bounded sequences and $\{\mu(k,A)\}_{k\geq0}$ is the sequence of singular values of the compact operator $A\in B(H)$ taken in the descending order. This trace is finite on $0\leq A\in B(H)$ if $A$ belongs to the Marcinkiewicz ideal (see e.g. \cite{GohKr1},\cite{GohKr2},\cite{Pietsch})
$$\mathcal{M}_{\psi}:=\{A\in B(H):\ \sup_{n\geq0}\frac1{\psi(n+1)}\sum_{k=0}^n\mu(k,A)<\infty\}.$$
In \cite{KSS}, Dixmier's result and construction was extended to an arbitrary Marcin\-kiewicz ideal $\mathcal{M}_{\psi}$ with the following condition on $\psi$
\begin{equation}\label{dpss psi cond}
\liminf_{t\to\infty}\frac{\psi(2t)}{\psi(t)}=1.
\end{equation}

All the traces defined above by formula \eqref{dixmier trace} vanish on the ideal $\mathcal{L}_1$ consisting of all compact operators $A\in B(H)$ such that
$${\rm Tr}(|A|)\stackrel{def}{=}\sum_{k=0}^\infty\mu(k,A)<\infty.$$

A symmetric operator ideal $\mathcal{E}\subset B(H)$ is a Banach space such that $A\in\mathcal{E}$ and $\{\mu(k,B)\}_{k\geq0}\leq\{\mu(k,A)\}_{k\geq0}$ implies that $B\in\mathcal{E}$ and $\|B\|_{\mathcal{E}}\leq\|A\|_{\mathcal{E}}$ (see e.g. \cite{GohKr1}, \cite{GohKr2}, \cite{Simon}\footnote{ We have to caution the reader that in Theorem 1.16 of \cite{Simon} the assertion $(b)$ does not hold for the norm of an arbitrary symmetric operator ideal $\mathcal {E}$ (see e.g. corresponding counterexamples in \cite[p. 83]{KS}).}, \cite{Schatten}, \cite{KScan}).

In analyzing Dixmier's proof of the linearity of $\tau_{\omega}$ given by \eqref{dixmier trace}, it was observed in \cite{KSS} (see also \cite{CaSuk,DPSS}) that $\tau_{\omega}$ possesses the following fundamental property, namely if $0\leq A,B\in\mathcal{M}_{\psi}$ are such that
\begin{equation}\label{hl maj disk}
\sum_{k=0}^n\mu(k,B)\leq\sum_{k=0}^n\mu(k,A),\quad\forall n\geq0,
\end{equation}
then $\tau_{\omega}(B)\leq\tau_{\omega}(A).$ Such a class of traces was termed \lq\lq fully symmetric\rq\rq in \cite{KScan}, \cite{SZ} (see also earlier papers \cite{DPSS},\cite{LSS}, where the term \lq\lq symmetric\rq\rq was used). It is natural to consider such traces only on fully symmetric operator ideals $\mathcal{E}$ (that is, on symmetric operator ideals $\mathcal{E}$ satisfying the condition: if $A,B$ satisfy \eqref{hl maj disk} and $A\in\mathcal{E},$ then $B\in\mathcal{E}$ and $\|B\|_{\mathcal{E}}\leq\|A\|_{\mathcal{E}}$). In fact, it was established in \cite{DPSS} that every Marcinkiewicz ideal $\mathcal{M}_{\psi}$ with $\psi$ satisfying the condition \eqref{dpss psi cond} possesses fully symmetric traces.

Furthermore, in the recent paper \cite{KSS}, the following unexpected result was established.

\begin{thm}\label{kss theorem} If $\psi$ satisfies the condition \eqref{dpss psi cond}, then every fully symmetric trace on $\mathcal{M}_{\psi}$ is a Dixmier trace $\tau_{\omega}$ for some $\omega.$
\end{thm}

In his seminal paper \cite{Wodzicki}, Wodzicki considered multiplicative renormalisation of positive compact operators. He was probably, the first who suggested that Dixmier construction works on the symmetric operator ideals different from Marcinkiewicz ideals.

More precisely, given a positive function $\psi$ on $(0,\infty)$ (Wodzicki did not assume this function to be either increasing or concave), one can construct a mapping
$$A\to\left\{\frac1{\psi(n+1)}\sum_{k=0}^n\mu(k,A)\right\}_{n\geq0}.$$
Applying some limiting procedure to the latter sequence (Wodzicki used Stone-\v{C}ech compactification for this purpose), we are left with a question whether this construction produces a trace. If it does, then it is natural to refer to such a trace as to the Dixmier trace. 

Wodzicki proved (a very complicated) criterion for the additivity of multiplicative renormalisation (see Theorem 3.4 in \cite{Wodzicki}). The questions of finiteness and non-triviality, also considered in \cite{Wodzicki},  happen to be even harder. In fact, it is proved in \cite{Wodzicki} that, for every principal ideal, multiplicative renormalisation produces a trace if and only if the ideal admits a trace. The latter result still relates to the realm of Marcinkiewicz operator ideals and can be compared with \cite{Varga,DPSS,GI,KSS}.

Our main result extends the above mentioned results of Wodzicki and Theorem \ref{kss theorem} to an arbitrary fully symmetric operator ideal.

\begin{thm}\label{main theorem in paper} Let $\mathcal{E}$ be a fully symmetric operator ideal. If $\mathcal{E}$ admits a trace, then there are Dixmier traces on $\mathcal{E}.$ Moreover, those Dixmier traces are weak$^*$ dense in the set of all fully symmetric traces on $\mathcal{E}.$
\end{thm}

The result of Theorem \ref{main theorem in paper} is a combination of Theorems \ref{density} and \ref{main theorem} below.

It should also be pointed out that the result of Theorem \ref{linearity criterion} below substantially strengthens the result of Theorem 3.4 in \cite{Wodzicki} and is much easier to apply in concrete situations (at least, in the setting of symmetric operator ideals).

\section{Preliminaries}

The theory of singular traces on symmetric operator ideals rests on some classical analysis which we now review for completeness. For more information, we refer to \cite{SZcrelle}.

Let $H$ be a Hilbert space and let $B(H)$ be the algebra of all bounded operators on $H$ equipped with the uniform norm. For every $A\in B(H),$ one can define a singular value function $\mu(A)$ (see e.g. \cite{FackKosaki}).
\begin{defi} Given an operator $A\in B(H),$ its singular value function $\mu(A)$ is defined by the formula
$$\mu(t,A)=\inf\{\|Ap\|:\ {\rm Tr}(1-p)\leq t\}.$$
\end{defi}
Clearly, $\mu(A)$ is a step function and, therefore, it can be identified with the sequence of singular numbers of the operators $A$ (the singular values are the eigenvalues of the operator $|A|=(A^*A)^{1/2}$ arranged with multiplicity in decreasing order). That is, we also use the notation $\mu(A)=\{\mu(k,A)\}_{k\geq0}.$

Equivalently, $\mu(A)$ can be defined in terms of the distribution function $d_A$ of $A.$ That is, setting
$$d_A(s)={\rm Tr}(E^{|A|}(s,\infty)),\quad s\geq0,$$
we obtain
$$\mu(t,A)=\inf\{s\geq0:\ d_A(s)\leq t\},\quad t>0.$$
Here, $E^{|A|}$ denotes the spectral projection of the operator $|A|.$

Further, we need to recall the important notion of Hardy--Littlewood majorization.
\begin{defi} The operator $B\in B(H)$ is said to be majorized by the operator $A\in B(H)$ (written $B\prec\prec A$) if and only if
$$\int_0^t\mu(s,B)ds\leq\int_0^t\mu(s,A)ds,\quad t\geq0.$$
\end{defi}
We have (see \cite{FackKosaki})
\begin{equation}\label{submaj ineq}
A+B\prec\prec\mu(A)+\mu(B)\prec\prec 2\sigma_{1/2}\mu(A+B)
\end{equation}
for every positive operators $A,B\in B(H).$

If $s>0,$ the dilation operator $\sigma_s:L_{\infty}(0,\infty)\to L_{\infty}(0,\infty)$ is defined by setting
$$(\sigma_s(x))(t)=x(\frac{t}{s}),\quad t>0.$$
Similarly, in the sequence case, we define an operator $\sigma_n$ by setting
$$\sigma_n(a_0,a_1,\cdots)=(\underbrace{a_0,\cdots,a_0}_{\mbox{$n$ times}},\underbrace{a_1,\cdots,a_1}_{\mbox{$n$ times}},\cdots)$$
and an operator $\sigma_{1/2}$ by setting
$$\sigma_{1/2}:(a_0,a_1,a_2,a_3,a_4,\cdots)\to(\frac{a_0+a_1}2,\frac{a_2+a_3}2,\cdots).$$

Below, we define symmetric ideals of $l_{\infty}$ and that of $B(H).$

\begin{defi} An ideal $E$ of the algebra $l_{\infty},$ equipped with the norm $\|\cdot\|_E,$ is said to be symmetric if
\begin{enumerate}
\item $(E,\|\cdot\|_E)$ is a Banach space.
\item For every $x\in E$ and every $y\in l_{\infty},$ we have $\|xy\|_E\leq\|x\|_E\|y\|.$
\item For every $x\in E$ and every permutation $\pi:\mathbb{Z}_+\to\mathbb{Z}_+,$ we have $x\circ\pi\in E$ and $\|x\circ\pi\|_E=\|x\|_E.$
\end{enumerate}
\end{defi}

\begin{defi} A two-sided ideal $\mathcal{E}$ of the algebra $B(H),$ equipped with the norm $\|\cdot\|_{\mathcal{E}},$ is called symmetric operator ideal if
\begin{enumerate}
\item $(\mathcal{E},\|\cdot\|_{\mathcal{E}})$ is a Banach space.
\item For every $A\in\mathcal{E}$ and every $B\in B(H),$ we have
$$\|AB\|_{\mathcal{E}}\leq\|A\|_{\mathcal{E}}\|B\|,\quad \|BA\|_{\mathcal{E}}\leq\|A\|_{\mathcal{E}}\|B\|.$$
\end{enumerate}
\end{defi}

It follows easily from the definition that, for every $A\in\mathcal{E}$ and every $B\in B(H)$ with $\mu(B)\leq\mu(A),$ we have $B\in\mathcal{E}$ and $\|B\|_{\mathcal{E}}\leq\|A\|_{\mathcal{E}}$ (see, e.g. \cite{GohKr1}). In particular, we have $\|A\|_{\mathcal{E}}=\|UA\|_{\mathcal{E}}=\|AU\|_{\mathcal{E}}$ for every $A\in\mathcal{E}$ and every unitary $U\in B(H).$ Thus, $\|A\|_{\mathcal{E}}=\|A^*\|_{\mathcal{E}}=\|\,|A|\,\|_{\mathcal{E}}$ for every $A\in\mathcal{E}.$ It is well-known that every proper ideal of the algebra $B(H)$ consists of compact operators.

The following fundamental result appeared in \cite{KS}.

\begin{thm} Let $E$ be a symmetric ideal in $l_{\infty}.$ The set
$$\mathcal{E}=\{A\in B(H):\ \mu(A)\in E\}$$
equipped with a norm $\|A\|_{\mathcal{E}}=\|\mu(A)\|_E$ is a symmetric operator ideal.
\end{thm}

\begin{defi} Symmetric operator ideal is said to be fully symmetric if, for every operator $A\in\mathcal{E}$ and $B\in B(H)$ such that $B\prec\prec A,$ we have $B\in\mathcal{E}$ and $\|B\|_{\mathcal{E}}\leq\|A\|_{\mathcal{E}}.$
\end{defi}

One should note that every fully symmetric operator ideal is a union of Marcin\-kiewicz operator ideals (see the text following Theorem II.5.7 of \cite{KPS}).

\begin{defi} Let $\mathcal{E}$ be a symmetric operator ideal. A linear functional $\varphi:\mathcal{E}\to\mathbb{C}$ is said to be a trace if $\varphi(U^{-1}AU)=\varphi(A)$ for every $A\in\mathcal{E}$ and every unitary $U\in B(H).$
\end{defi}

One can show that $\varphi(A)=\varphi(B)$ for every positive operators $A,B\in\mathcal{E}$ such that $\mu(A)=\mu(B).$

\begin{defi} A trace $\varphi:\mathcal{E}\to\mathbb{C}$ is called fully symmetric if $\varphi(B)\leq\varphi(A)$ for every positive operators $A,B\in\mathcal{E}$ with $B\prec\prec A.$
\end{defi}

\section{Relatively normal traces are dense}

In this section, we introduce an important class of relatively normal traces  (see Definition \ref{relatively normal def} below) and prove that they are weak$^*$ dense among all fully symmetric traces. The main result of this section is Theorem \ref{density}.

Let $\mathcal{E}$ be a fully symmetric operator ideal and let $\varphi$ be a fully symmetric trace on $\mathcal{E}.$ In what follows, $\mathcal{E}_+$ denotes the positive cone of $\mathcal{E}.$




\begin{lem} Let $\mathcal{M}_{\psi}\subset\mathcal{E}$ be a Marcinkiewicz space and let $\varphi$ be a fully symmetric trace on $\mathcal{E}.$ The mapping $\varphi_{normal,\psi}:\mathcal{E}_+\to\mathbb{R}$ defined by setting
\begin{equation}\label{normal part definition}
\varphi_{normal,\psi}(A)=\sup\{\varphi(B):\ B\in\mathcal{M}_{\psi},\ 0\leq B\prec\prec A\},\quad 0\leq A\in\mathcal{E},
\end{equation}
is additive on the positive cone of $\mathcal{E}.$
\end{lem}
\begin{proof} Let $A_1,A_2\in\mathcal{E}_+.$ Let $B\in\mathcal{M}_{\psi}$ be such that $0\leq B\prec\prec A_1+A_2.$ By \cite[Theorem 2.2]{DDP}, there exists a linear operator $\mathbf{C}:B(H)\to B(H)$ (a positive contraction both in $B(H)$ and in $\mathcal{L}_1$) such that $B=\mathbf{C}(A_1+A_2).$ Setting $B_1=\mathbf{C}(A_1)\geq0$ and $B_2=\mathbf{C}(A_2)\geq0,$ we have $B=B_1+B_2.$ Therefore, $0\leq B_i\leq B\in\mathcal{M}_{\psi}$ and $B_i\prec\prec A_i.$ Hence, by definition \eqref{normal part definition},
$$\varphi(B)=\varphi(B_1)+\varphi(B_2)\leq\varphi_{normal,\psi}(A_1)+\varphi_{normal,\psi}(A_2).$$
Taking the supremum over all $B$ in question, we obtain
\begin{equation}\label{nfi sub}
\varphi_{normal,\psi}(A_1+A_2)\leq\varphi_{normal,\psi}(A_1)+\varphi_{normal,\psi}(A_2).
\end{equation}

Fix $\varepsilon>0.$ There exist $B_i\in\mathcal{M}_{\psi}$ such that $0\leq B_i\prec\prec A_i$ and $\varphi(B_i)>\varphi_{normal,\psi}(A_i)-\varepsilon.$ In particular, we have
\begin{equation}\label{nfi1}
\varphi_{normal,\psi}(A_1)+\varphi_{normal,\psi}(A_2)\leq 2\varepsilon+\varphi(B_1+B_2).
\end{equation}
Further, we have
$$B_1+B_2\prec\prec\mu(B_1)+\mu(B_2)\prec\prec\mu(A_1)+\mu(A_2)\prec\prec2\sigma_{1/2}\mu(A_1+A_2).$$
It follows from \eqref{nfi1} and definition \eqref{normal part definition} that
\begin{align*}
\varphi_{normal,\psi}(A_1)+\varphi_{normal,\psi}(A_2)\leq\\
\leq2\varepsilon+\varphi_{normal,\psi}(2\sigma_{1/2}\mu(A_1+A_2))=2\varepsilon+\varphi_{normal,\psi}(A_1+A_2).
\end{align*}
Here, the last equality follows from Lemma \ref{sim dil rem} below. Since $\varepsilon$ is arbitrarily small, we have
\begin{equation}\label{nfi super}
\varphi_{normal,\psi}(A_1)+\varphi_{normal,\psi}(A_2)\leq\varphi_{normal,\psi}(A_1+A_2).
\end{equation}
The assertion follows from \eqref{nfi sub} and \eqref{nfi super}.
\end{proof}

It is proved in the following lemma that $\varphi_{normal,\psi}$ can be viewed as a \lq\lq normal part\rq\rq of the trace $\varphi$ with respect to the subspace $\mathcal{M}_{\psi}.$

\begin{lem}\label{marcinkiewicz normal remark} The mapping $\varphi_{normal,\psi}:\mathcal{E}_+\to\mathbb{R}$ extends to a fully symmetric trace on $\mathcal{E}.$ Moreover, $\varphi=\varphi_{normal,\psi}$ on $\mathcal{M}_{\psi}$ and $(\varphi_{normal,\psi})_{normal,\psi}=\varphi_{normal,\psi}$ on $\mathcal{E}.$
\end{lem}
\begin{proof} Every additive functional on $\mathcal{E}_+$ uniquely extends to a linear functional on $\mathcal{E}.$ In particular, so does $\varphi_{normal,\psi}:\mathcal{E}_+\to\mathbb{R}.$

Let $A_1,A_2\in\mathcal{E}$ be positive operators such that $A_2\prec\prec A_1.$ It follows that
$$\{\varphi(B):\ B\in\mathcal{M}_{\psi},\ 0\leq B\prec\prec A_2\}\subset \{\varphi(B):\ B\in\mathcal{M}_{\psi},\ 0\leq B\prec\prec A_1\}.$$
Therefore, $\varphi_{normal,\psi}(A_2)\leq\varphi_{normal,\psi}(A_1).$ Hence, $\varphi_{normal,\psi}$ is a fully symmetric trace on $\mathcal{E}.$

The second assertion is obvious. In order to prove the third assertion, fix a positive operator $A\in\mathcal{E}.$ By definition, $(\varphi_{normal,\psi})_{normal,\psi}(A)\leq\varphi_{normal,\psi}(A).$ Select $B_m\in\mathcal{M}_{\psi}$ such that $0\leq B_m\prec\prec A$ and such that $\varphi(B_m)\to\varphi_{normal,\psi}(A).$ Clearly, $\varphi_{normal,\psi}(B_m)=\varphi(B_m).$ Thus, $\varphi_{normal,\psi}(B_m)\to\varphi_{normal,\psi}(A).$ Therefore, $(\varphi_{normal,\psi})_{normal,\psi}(A)\geq\varphi_{normal,\psi}(A),$ and the third assertion is proved.
\end{proof}

\begin{defi}\label{relatively normal def} A fully symmetric trace $\varphi$ on $\mathcal{E}$ is called relatively normal if there exists a Marcinkiewicz space $\mathcal{M}_{\psi}\subset\mathcal{E}$ such that $\varphi=\varphi_{normal,\psi}.$
\end{defi}

\begin{thm}\label{density} Relatively normal traces on $\mathcal{E}$ are weak$^*$ dense in the set of all fully symmetric traces on $\mathcal{E}.$
\end{thm}
\begin{proof} The set $E_+^{\downarrow}=\{\mu(A),\ A\in\mathcal{E}\},$ equipped with the partial ordering given by the Hardy-Littlewood majorization is a directed set. For every $x\in E_+^{\downarrow},$ let $\psi_x:\mathbb{R}_+\to\mathbb{R}_+$ be a concave increasing function such that $\psi_x'=x.$ For every given fully symmetric functional $\varphi,$ consider the net $\{\varphi_{n,\psi_x}\in\mathcal{E}^*,\ x\in E_+^{\downarrow}\}.$ We claim that this net weak$^*$ converges to the functional $\varphi.$

Recall that the base of weak$^*$ topology (that is, $\sigma(\mathcal{E}^*,\mathcal{E})$) is formed by the sets
$$N(A_1,\cdots,A_m,\varepsilon)=\{\theta\in\mathcal{E}^*:\ |\theta(A_k)|<\varepsilon,\ 1\leq k\leq m\}.$$

Fix some neighborhood $U$ of $0$ in the weak$^*$ topology. Select $\varepsilon>0$ and operators $0\leq A_k\in\mathcal{E}$ such that
$$N(A_1,\cdots,A_m,\varepsilon)\subset U.$$
Set $y=\sum_{k=1}^m\mu(A_k).$ It is clear that, for every $x\in E_+^{\downarrow}$ such that $y\prec\prec x,$ we have $A_k\in\mathcal{M}_{\psi_x}.$ It follows from Lemma \ref{marcinkiewicz normal remark} that $(\varphi_{n,\psi_x}-\varphi)(A_k)=0.$ Therefore,
$$\varphi-\varphi_{n,\psi_x}\in \{\theta\in\mathcal{E}^*:\ |\theta(A_k)|=0,\ 1\leq k\leq m\}\subset U.$$
\end{proof}

\section{Dixmier traces}

In this section, we introduce the concept of Dixmier trace on symmetric operator ideals. The main result of this section is Theorem \ref{linearity criterion}.

\subsection{Extension of states}

As usual, a state on the algebra $l_{\infty}$ is a positive linear functional $\omega$ such that $\omega(1)=1.$ A state $\omega$ is called singular if it vanishes on all finitely supported sequences. A state $\omega$ is called dilation invariant if $\omega\circ\sigma_n=\omega,$ $n\in\mathbb{N}.$

\begin{lem}\label{extension lemma} Every state $\omega$ on the algebra $l_{\infty}$ admits an extension to an additive mapping $\omega$ from the set $\mathfrak{L}_+$ of the positive (unbounded) sequences to $\mathbb{R}_+\cup\{\infty\}.$ This extension is defined by setting
\begin{equation}\label{omega extension}
\omega(x)=\sup\{\omega(y),\ 0\leq y\leq x,\ y\in l_{\infty}\},\quad 0\leq x\in\mathfrak{L}_+.
\end{equation}
\end{lem}
\begin{proof} Let $x_1,x_2\in\mathfrak{L}_+.$ If $0\leq y\in l_{\infty}$ is such that $y\leq x_1+x_2,$ then there exist positive elements $y_1,y_2\in l_{\infty}$ such that $y=y_1+y_2,$ $y_1\leq x_1$ and $y_2\leq x_2.$ It follows from \eqref{omega extension} that
$$\omega(y)=\omega(y_1)+\omega(y_2)\leq\omega(x_1)+\omega(x_2).$$
Taking the supremum over all such $y,$ we obtain
\begin{equation}\label{omega sub}
\omega(x_1+x_2)\leq\omega(x_1)+\omega(x_2).
\end{equation}

Now, we prove the converse inequality. The latter becomes trivial if $\omega(x_1)=\infty$ or $\omega(x_2)=\infty.$ Thus, we may assume without loss of generality that both $\omega(x_1)<\infty$ and $\omega(x_2)<\infty.$ Fix $\varepsilon>0.$ Let $y_i\in l_{\infty}$ be such that $0\leq y_i\leq x_i$ and $\omega(y_i)>\omega(x_i)-\varepsilon.$ It follows from \eqref{omega extension} that
$$\omega(x_1)+\omega(x_2)\leq 2\varepsilon+\omega(y_1)+\omega(y_2)=2\varepsilon+\omega(y_1+y_2)\leq 2\varepsilon+\omega(x_1+x_2).$$
Since $\varepsilon$ is arbitrarily small, we obtain
\begin{equation}\label{omega super}
\omega(x_1)+\omega(x_2)\leq\omega(x_1+x_2).
\end{equation}
The assertion follows from \eqref{omega sub} and \eqref{omega super}.
\end{proof}

It follows directly from the definition \eqref{omega extension} that the extension $\omega:\mathfrak{L}_+\to\mathbb{R}\cup\{\infty\}$ defined in Lemma \ref{extension lemma} is dilation invariant if and only if $\omega:l_{\infty}\to\mathbb{R}$ is dilation invariant.

\begin{lem}\label{first simple lemma} For every state $\omega$ on the algebra $l_{\infty}$ and for every $x\in\mathfrak{L}_+,$ we have $$\omega(\min\{n,x\})\to\omega(x)$$
as $n\to\infty.$
\end{lem}
\begin{proof} Fix a sequence $\{x_n\}_{n\geq0}\subset l_{\infty}$ such that $x_n\leq x$ for every $n\geq0,$ and such that $\omega(x_n)\to\omega(x)$ as $n\to\infty.$ Evidently, $x_n\leq\min\{\|x_n\|_{\infty},x\}\leq x.$ It follows that $\omega(\min\{\|x_n\|_{\infty},x\})\to\omega(x)$ as $n\to\infty.$ If $\|x_n\|_{\infty}\to\infty,$ then we conclude the proof. If $\|x_n\|_{\infty}\leq C$ for $n\geq0,$ then $\omega(x)=\omega(\min\{n,x\})$ for every $n\geq C$ and the assertion follows.
\end{proof}

\begin{lem}\label{simple lemma} Let $\omega$ be a state on the algebra $l_{\infty}.$ Let $0\leq z\in l_{\infty}$ be such that $\omega(z)=0$ and let $u\in\mathfrak{L}_+$ be such that $\omega(u)<\infty.$ It follows that $\omega(uz)=0.$
\end{lem}
\begin{proof} It is clear that $u=\min\{n,u\}+(u-n)_+.$ It follows from  Lemma \ref{first simple lemma} that the sequence $\omega(\min\{n,u\})$ converges to $\omega(u).$ Since $\omega(u)<\infty,$ it follows that $\omega((u-n)_+)\to0.$ On the other hand, we have $uz\leq nz+\|z\|_{\infty}(u-n)_+.$ Since $\omega(z)=0,$ it follows that $\omega(uz)\leq\|z\|_{\infty}\omega((u-n)_+).$ Passing $n\to\infty,$ we conclude the proof.
\end{proof}

\subsection{Dixmier traces}

\begin{defi}\label{dixmier trace definition} Let $\mathcal{E}$ be a fully symmetric operator ideal. For a given concave increasing function $\psi$ with $\mathcal{M}_{\psi}\subset\mathcal{E}$ and given dilation invariant singular state $\omega$ on $l_{\infty}$  define a mapping $\tau_{\omega}:\mathcal{E}_+\to\mathbb{R}_+\cup\{\infty\}$ by setting
$$\tau_{\omega}(A)=\omega(\frac1{\psi(n+1)}\sum_{k=0}^n\mu(k,A)),\quad0\leq A\in\mathcal{E},$$
where the extension of $\omega$ to $\mathcal{L}_+$ is given by Lemma \ref{extension lemma}. If the mapping $\tau_{\omega}$ is finite and additive on $\mathcal{E}_+,$ then its linear extension to $\mathcal{E}$ is called a Dixmier trace on $\mathcal{E}.$
\end{defi}

\begin{lem}\label{sim dil rem} Let $\mathcal{E}$ be a symmetric operator ideal and let $\varphi$ be a trace on $\mathcal{E}.$ For every positive $A\in\mathcal{E},$ we have $\varphi(2\sigma_{1/2}\mu(A))=\varphi(A).$
\end{lem}
\begin{proof} Without loss of generality, we can take $A=\mu(A).$ Thus,
$$A={\rm diag}(\mu(0,A),0,\mu(2,A),0,\cdots)+{\rm diag}(0,\mu(1,A),0,\mu(3,A),\cdots).$$
Similarly,
$$2\sigma_{1/2}\mu(A)={\rm diag}(\mu(0,A),\mu(2,A),\cdots)+{\rm diag}(\mu(1,A),\mu(3,A),\cdots).$$
The assertion follows immediately.
\end{proof}

The following theorem gives a necessary and sufficient condition for the mapping $\tau_{\omega}$ to be a trace on $\mathcal{E}.$

\begin{thm}\label{linearity criterion} Let $\mathcal{E}$ be a symmetric operator ideal and let $\mathcal{M}_{\psi}\subset\mathcal{E}.$ Let $\omega$ be a singular state on the algebra $l_{\infty}$ such that $\tau_{\omega}$ is finite on $\mathcal{E}_+.$ The mapping $\tau_{\omega}$ is additive on $\mathcal{E}_+$ if and only if
\begin{equation}\label{linear if}
\omega(\frac{\psi(2n+1)}{\psi(n+1)})=1.
\end{equation}
\end{thm}
\begin{proof} Note that concave function $\psi$ is subadditive and, therefore,
$$\{\frac{\psi(2n+1)}{\psi(n+1)}\}_{n\geq0}\in l_{\infty}.$$

Suppose first that $\tau_{\omega}$ is additive on $\mathcal{E}_+.$ Set
$$A={\rm diag}(\psi(1),\psi(2)-\psi(1),\psi(3)-\psi(2),\cdots).$$
Note that $A\in\mathcal{M}_{\psi}\subset\mathcal{E}.$ It is obvious from the definition of $A$ that
$$\omega(\frac{\psi(2n+1)}{\psi(n+1)})=\tau_{\omega}(2\sigma_{1/2}\mu(A)).$$
The equality \eqref{linear if} follows now from Lemma \ref{sim dil rem}.

Assume now that the equality \eqref{linear if} holds. Let $A,B\in\mathcal{E}$ be positive operators. It follows from the left hand side inequality in \eqref{submaj ineq} that
\begin{equation}\label{sssub}
\tau_{\omega}(A+B)\leq\tau_{\omega}(A)+\tau_{\omega}(B).
\end{equation}
In order to prove converse inequality, introduce the positive sequence 
$$z=\left\{1-\frac{\psi(n+1)}{\psi(2n+1)}\right\}_{n\geq0}.$$
By the assumption, we have $\omega(z)=0.$ By Definition \ref{dixmier trace definition} and Lemma \ref{simple lemma}, we have
$$\tau_{\omega}(A)+\tau_{\omega}(B)=\omega(\frac1{\psi(n+1)}\sum_{k=0}^n\mu(k,A)+\mu(k,B))=$$
$$=\omega((1-z_n)\frac1{\psi(n+1)}\sum_{k=0}^n\mu(k,A)+\mu(k,B))=\omega(\frac1{\psi(2n+1)}\sum_{k=0}^n\mu(k,A)+\mu(k,B)).$$
Applying now the right hand side inequality in \eqref{submaj ineq}, we obtain
$$\tau_{\omega}(A)+\tau_{\omega}(B)\leq\omega(\frac1{\psi(2n+1)}\sum_{k=0}^{2n+1}\mu(k,A+B)).$$
Since $\omega$ is dilation invariant, it follows that
$$\tau_{\omega}(A)+\tau_{\omega}(B)\leq(\omega\circ\sigma_2)(\frac1{\psi(2n+1)}\sum_{k=0}^{2n+1}\mu(k,A+B)).$$
However,
$$\sigma_2\left\{\frac1{\psi(2n+1)}\sum_{k=0}^{2n+1}\mu(k,A+B)\right\}_{n\geq0}\in\left\{\frac1{\psi(n+1)}\sum_{k=0}^n\mu(k,A+B)\right\}_{n\geq0}+c_0.$$
Since $\omega|_{c_0}=0,$ it follows that
\begin{equation}\label{sssuper}
\tau_{\omega}(A)+\tau_{\omega}(B)\leq\tau_{\omega}(A+B).
\end{equation}
The assertion follows now from \eqref{sssub} and \eqref{sssuper}.
\end{proof}

\section{Relatively normal traces are Dixmier traces}

In this section we prove that every relatively normal trace on symmetric operator ideal must be the Dixmier trace. The main result of this section is Theorem \ref{main theorem}.

Let $A,B\in B(H)$ be positive operators. Let $A\wedge B$ be any positive operator from $B(H)$ such that
\begin{equation}\label{wedge}
\sum_{k=0}^n\mu(k,A\wedge B)=\min\{\sum_{k=0}^n\mu(k,A),\sum_{k=0}^n\mu(k,B)\},\quad n\geq0.
\end{equation}

\begin{lem}\label{normalnost gamma} Let $\mathcal{E}$ be a fully symmetric operator ideal and let $\varphi$ be a relatively normal fully symmetric trace on $\mathcal{E}.$ There exists a positive operator $B\in\mathcal{E}$ such that
$$\varphi(A)=\lim_{n\to\infty}\varphi(A\wedge nB),\quad 0\leq A\in\mathcal{E}.$$
\end{lem}
\begin{proof} By assumption and Definition \ref{relatively normal def}, there exists a Marcinkiewicz subspace $\mathcal{M}_{\psi}\subset\mathcal{E}$ such that $\varphi=\varphi_{normal,\psi}.$ Set
$$B={\rm diag}(\psi(1),\psi(2)-\psi(1),\psi(3)-\psi(2),\cdots).$$
Obviously, $B\in\mathcal{M}_{\psi}\subset\mathcal{E}.$ For every positive $A\in\mathcal{E},$ we have
\begin{align*}
\varphi(A)=\varphi_{normal,\psi}(A)=\sup\{\varphi(C):\ C\in\mathcal{ M}_{\psi},\ 0\leq C\prec\prec A\}=\\
=\lim_{n\to\infty}\sup\{\varphi(C):\ \|C\|_{\mathcal{M}_{\psi}}\leq n,\ 0\leq C\prec\prec A\}.
\end{align*}
It follows now from the definition of Marcinkiewicz operator ideal that
\begin{align*}
\varphi(A)=\lim_{n\to\infty}\sup\{\varphi(C):\ C\prec\prec nB,\ 0\leq C\prec\prec A\}=\\
=\lim_{n\to\infty}\sup\{\varphi(C):\ 0\leq C\prec\prec A\wedge nB\}.
\end{align*}
Since the trace $\varphi$ is fully symmetric, it follows that
$$\varphi(A)=\lim_{n\to\infty}\varphi(A\wedge nB).$$
\end{proof}

\begin{thm}\label{main theorem} Let $\mathcal{E}$ be a fully symmetric operator ideal and let $\varphi$ be a relatively normal (with respect to the Marcinkiewicz space $\mathcal{M}_{\psi}\subset\mathcal{E}$) fully symmetric trace on $\mathcal{E}.$ There exists a dilation invariant singular state $\omega$ on $l_{\infty}$ (extended to $\mathcal{L}_+$ by Lemma \ref{extension lemma}) such that
$$\varphi(A)=\omega(\frac1{\psi(n+1)}\sum_{k=0}^n\mu(k,A)),\quad 0\leq A\in \mathcal{E}.$$
\end{thm}
\begin{proof} The functional $\varphi|_{\mathcal{M}_{\psi}}$ is fully symmetric. By Theorem \ref{kss theorem}, $\varphi|_{\mathcal{M}_{\psi}}$ is a Dixmier trace. In particular, there exists a dilation invariant singular state $\omega$ on $l_{\infty}$ such that
\begin{equation}\label{kss result}
\varphi(C)=\omega(\frac1{\psi(n+1)}\sum_{k=0}^n\mu(k,C)),\quad 0\leq C\in \mathcal{M}_{\psi}.
\end{equation}
For every positive $A\in\mathcal{E},$ define $\mathbf{T}(A)\in\mathfrak{L}_+$ by setting
$$\mathbf{T}(A)=\left\{\frac1{\psi(n+1)}\sum_{k=0}^n\mu(k,A)\right\}_{n\geq0}.$$
Set
$$B={\rm diag}(\psi(1),\psi(2)-\psi(1),\psi(3)-\psi(2),\cdots).$$
Obviously, $B\in\mathcal{M}_{\psi}\subset\mathcal{E}.$ Dividing the equality \eqref{wedge} by $\psi(n+1)$ and taking into account that
$$\frac1{\psi(n+1)}\sum_{k=0}^n\mu(k,B)=1,$$
we obtain $\mathbf{T}(A\wedge nB)=\min\{\mathbf{T}(A),n\}.$ It follows from \eqref{kss result} and Lemma \ref{normalnost gamma} that
$$\varphi(A)=\lim_{n\to\infty}\omega(\min\{\mathbf{T}(A),n\}),\quad 0\leq A\in\mathcal{E}.$$
By Lemma \ref{first simple lemma}, we conclude that $\varphi(A)=\omega(\mathbf{T}(A))$ for every $0\leq A\in\mathcal{E}.$
\end{proof}

\section{Wodzicki representation of Dixmier traces}

In this section, we prove that every relatively normal trace on a symmetric operator ideal can be represented in the form proposed by Wodzicki \cite{Wodzicki}. The main result of this section is Theorem \ref{wodzicki representation}.

The Banach space $l_{\infty}$ is a commutative $C^*$-algebra. Let $\beta\mathbb{N}$ be the set of all nontrivial homomorphic functionals on $l_{\infty}.$ Clearly, $\beta\mathbb{N}$ is a weak$^*$ closed subset of a unit ball of $l_{\infty}^*.$ By the Banach-Alaoglu theorem, the unit ball of $l_{\infty}^*$ (and, therefore, the set $\beta\mathbb{N}$) is weak$^*$ compact. By Gelfand-Naimark theorem, $l_{\infty}$ is isometrically isomorphic (via Gelfand transform) to the $C^*-$algebra of all continuous functions on $\beta\mathbb{N}.$ The set $\beta\mathbb{N}$ is usually called the Stone-\v{C}ech compactification of $\mathbb{N}.$ The set $\mathbb{N}_{\infty}=\beta\mathbb{N}\backslash\mathbb{N}$ is frequently referred to as to the set of all infinite integers.

\begin{lem}\label{dil invar lemma} Dilation semigroup $\sigma_n,$ $n\geq1,$ acts on $\mathbb{N}_{\infty}.$ Every dilation invariant singular state admits a representation
$$\omega(x)=\int_{\mathbb{N}_{\infty}}x(p)d\nu(p),\quad x\in l_{\infty}.$$
with $\nu$ being a finite regular dilation invariant Borel measure $\nu$ on $\mathbb{N}_{\infty}.$
\end{lem}
\begin{proof} Let $e_k=\{\delta_{kj}\}_{j\geq0},$ $k\geq0,$ be the standard basic sequence in $l_{\infty}.$ If $p\in\beta\mathbb{N},$ then
$$p(e_k)p(e_l)=p(e_ke_l)=0,\quad k\neq l.$$
In particular, at most one of the numbers $p(e_k),$ $k\geq0,$ is nonzero. Obviously, $p\in\mathbb{N}_{\infty}$ if and only if $p(e_k)=0$ for all $k\geq0.$

For every $p\in\beta\mathbb{N}$ and every $n\in\mathbb{N},$ the mapping $x\to(\sigma_nx)(p)$ is a homomorphism. Hence, it corresponds to a point $q\in\beta\mathbb{N}.$ If $q\notin\mathbb{N}_{\infty},$ then there exists $k\geq0$ such that $q(e_k)\neq0.$ Thus,
$$\sum_{m=kn}^{(k+1)n-1}p(e_k)=p(\sum_{m=kn}^{(k+1)n-1}e_k)=p(\sigma_ne_k)=q(e_k)\neq0.$$
Hence, $p(e_m)\neq0$ for some $kn\leq m<(k+1)n.$ Thus, $p\notin\mathbb{N}_{\infty}.$ It follows that $\sigma_n$ acts on $\mathbb{N}_{\infty}.$

By Riesz-Markov theorem (see \cite{ReedSimon}), for every state $\omega$ on $l_{\infty},$ there exists a finite regular Borel measure $\nu$ on $\beta\mathbb{N}$ such that
\begin{equation}\label{ciqq1}
\omega(x)=\int_{\beta\mathbb{N}}x(p)d\nu(p).
\end{equation}
Kakutani and Nakamura noted in \cite{Kakutani} that if the state $\omega$ is singular, then the measure $\nu$ is supported on $\mathbb{N}_{\infty}.$

For every $x\in l_{\infty},$ we have
\begin{equation}\label{ciqq2}
(\omega\circ\sigma_n)(x)=\int_{\mathbb{N}_{\infty}}(\sigma_nx)(p)d\nu(p)=\int_{\mathbb{N}_{\infty}}x(p)d(\nu\circ(\sigma_n)^{-1})(p).
\end{equation}
Since $\omega=\omega\circ\sigma_n$ for all $n\geq1,$ it follows from \eqref{ciqq1} and \eqref{ciqq2} that the measure $\nu$ is invariant with respect to the action of the dilation semigroup.
\end{proof}

\begin{cor}\label{int rep cor} Let $\omega$ be a dilation invariant singular state. There exists a finite regular dilation invariant Borel measure $\nu$ on $\mathbb{N}_{\infty}$ such that
$$\omega(x)=\int_{\mathbb{N}_{\infty}}x(p)d\nu(p)$$
for every $x\in\mathfrak{L}_+.$
\end{cor}
\begin{proof} Fix $p\in\mathbb{N}_{\infty}.$ Extend $p$ to an additive functional on $\mathfrak{L}_+$ by Lemma \ref{extension lemma}. For every $x\in\mathfrak{L}_+$ and for every $n\in\mathbb{N},$ we have $(\min\{x,n\})(p)=\min\{x(p),n\}.$

For a given $n\in\mathbb{N},$ it follows from above and from Lemma \ref{dil invar lemma} that
$$\omega(\min\{x,n\})=\int_{\mathbb{N}_{\infty}}\min\{x(p),n\}d\nu(p).$$
It follows from Levi theorem that
$$\int_{\mathbb{N}_{\infty}}\min\{x(p),n\}d\nu(p)\to \int_{\mathbb{N}_{\infty}}x(p)d\nu(p).$$
The assertion follows now from Lemma \ref{extension lemma}.
\end{proof}

\begin{thm}\label{wodzicki representation} Let $\mathcal{E}$ be a fully symmetric operator ideal and let $\varphi$ be a relatively normal (with respect to the Marcinkiewicz space $\mathcal{M}_{\psi}\subset\mathcal{E}$) fully symmetric trace on $\mathcal{E}.$ There exists a finite regular dilation invariant Borel measure $\nu$ on $\mathbb{N}_{\infty}$ such that
\begin{equation}\label{wodzicki representation of functionals}
\varphi(A)=\int_{\mathbb{N}_{\infty}}\left(\frac1{\psi(n+1)}\sum_{k=0}^n\mu(k,A)\right)(p)d\nu(p),\quad 0\leq A\in\mathcal{E}.
\end{equation}
\end{thm}
\begin{proof} The assertion follows immediately from Theorem \ref{main theorem} and Corollary \ref{int rep cor}.
\end{proof}

The following assertion is an easy corollary of Theorem \ref{wodzicki representation}.

\begin{cor} Let $\mathcal{E}$ be a fully symmetric operator ideal. One of the following mutually exclusive possibilities holds.
\begin{enumerate}
\item For every $A\in\mathcal{E},$ we have
$$\frac1n\|A^{\oplus n}\|_{\mathcal{E}}\to0,\quad\mbox{ as }n\to\infty.$$
\item There exist a concave increasing function $\psi$ and a finite regular dilation invariant Borel measure $\nu$ on $\mathbb{N}_{\infty}$ such that the mapping
$$A\to\int_{\mathbb{N}_{\infty}}\left(\frac1{\psi(n+1)}\sum_{k=0}^n\mu(k,A)\right)(p)d\nu(p),\quad 0\leq A\in\mathcal{E}$$
extends to a trace on $\mathcal{E}.$
\end{enumerate}
\end{cor}
\begin{proof} Suppose that there exists $A\in\mathcal{E}$ such that 
$$\frac1n\|A^{\oplus n}\|_{\mathcal{E}}\not\to0,\quad\mbox{ as }n\to\infty.$$
It follows from Theorem 5 (ii) in \cite{SZcrelle} that there exists a fully symmetric trace on $\mathcal{E}.$ By Theorem \ref{density}, it can be approximated by a relatively normal fully symmetric trace on $\mathcal{E}.$ The assertion follows immediately from Theorem \ref{wodzicki representation}.
\end{proof}


\begin{thebibliography}{100}
\bibitem{CaSuk} Carey A., Sukochev F. {\it Dixmier traces and some applications to noncommutative geometry.} (Russian)  Uspekhi Mat. Nauk {\bf 61} (2006),  no. 6 (372), 45--110; translation in Russian Math. Surveys {\bf 61} (2006),  no. 6, 1039--1099.
\bibitem{Connes} Connes A. {\it Noncommutative Geometry.} Academic Press, San Diego, 1994.
\bibitem{Connes-action} Connes A. {\it The action functional in noncommutative geometry.} Comm. Math. Phys. {\bf 117}  (1988),  no. 4, 673--683.
\bibitem{Dixmier} Dixmier J. {\it Existence de traces non normales.} (French)  C. R. Acad. Sci. Paris Ser. A-B  {\bf 262}  1966 A1107--A1108.
\bibitem{DDP} Dodds P., Dodds T., de Pagter B. {\it Fully symmetric operator spaces.} Integral Equations Operator Theory {\bf 15} (1992), no. 6, 942--972.
\bibitem{DPSSS} Dodds P., de Pagter B., Sedaev A., Semenov E., Sukochev F. {\it Singular symmetric functionals and Banach limits with additional invariance properties.} (Russian)  Izv. Ross. Akad. Nauk Ser. Mat. {\bf 67}  (2003), no. 6, 111--136;  translation in  Izv. Math. {\bf 67} (2003),  no. 6, 1187--1212.
\bibitem{DPSS} Dodds P., de Pagter B., Semenov E., Sukochev F. {\it Symmetric functionals and singular traces.}  Positivity {\bf 2}  (1998),  no. 1, 47--75.
\bibitem{FackKosaki} Fack T., Kosaki H. {\it Generalized $s$-numbers of $\tau-$measurable operators.} Pacific J. Math. {\bf 123} (1986), no. 2, 269--300.
\bibitem{GohKr1} Gohberg I., Krein M. {\it Introduction to the theory of linear nonselfadjoint operators.} Translations of Mathematical Monographs, Vol. 18 American Mathematical Society, Providence, R.I. 1969.
\bibitem{GohKr2} Gohberg I., Krein M. {\it Theory and applications of Volterra operators in Hilbert space.} Translations of Mathematical Monographs, Vol. 24 American Mathematical Society, Providence, R.I. 1970.
\bibitem{GI} Guido D., Isola T. {\it Singular traces on semifinite von Neumann algebras.} J. Funct. Anal. {\bf 134} (1995), no. 2, 451--485.
\bibitem{Kakutani} Kakutani S., Nakamura M. {\it Banach limits and the \v{C}ech compactification of a countable discrete set.}  Proc. Imp. Acad. Tokyo  {\bf 19},  (1943). 224--229.
\bibitem{KSS} Kalton N., Sedaev A., Sukochev F. {\it Fully symmetric functionals on a Marcinkiewicz space are Dixmier traces.} Adv. Math. {\bf 226} (2011) 3540--3549.
\bibitem{KS} Kalton N., Sukochev F. {\it Symmetric norms and  spaces of operators.} J.reine angew. Math.  (2008), 1--41
\bibitem{KScan} Kalton N., Sukochev F. {\it Rearrangement-invariant functionals with applications to traces on symmetrically normed ideals.} Canad. Math. Bull. {\bf 51} (2008), 67--80.
\bibitem{KPS} Krein S., Petunin Ju., Semenov E. {\it Interpolation of linear operators.} Nauka, Moscow, 1978 (in Russian); English translation in Translations of Math. Monographs, Vol. {\bf 54}, Amer. Math. Soc., Providence, RI, 1982.
\bibitem{LSS} Lord S., Sedaev A., Sukochev F. {\it Dixmier traces as singular symmetric functionals and applications to measurable operators.} J. Funct. Anal. {\bf 224} (2005),  no. 1, 72--106.
\bibitem{ReedSimon} Reed M., Simon B. {\it Methods of modern mathematical physics. I. Functional analysis.} Second edition. Academic Press, Inc. [Harcourt Brace Jovanovich, Publishers], New York, 1980.
\bibitem{Pietsch} Pietsch A. {\it About the Banach envelope of $l_{1,\infty}.$} Rev. Mat. Complut. {\bf 22} (2009),  no. 1, 209--226.
\bibitem{Schatten} Schatten R. {\it Norm ideals of completely continuous operators.} Second printing. Ergebnisse der Mathematik und ihrer Grenzgebiete, Band 27 Springer-Verlag, Berlin-New York 1970.
\bibitem{Simon} Simon B. {\it Trace ideals and their applications.} Second edition. Mathematical Surveys and Monographs, 120. American Mathematical Society, Providence, RI, 2005.
\bibitem{SZ} Sukochev F., Zanin D. {\it Orbits in symmetric spaces.} J.Funct.Anal. {\bf 257} (2009), no.1, 194-218.
\bibitem{SZcrelle} Sukochev F., Zanin D. {\it Traces on symmetric operator spaces.} to appear at J.reine angew. Math.
\bibitem{Varga} Varga J. {\it Traces on irregular ideals.} Proc. Amer. Math. Soc. {\bf 107} (1989), no. 3, 715--723.
\bibitem{Wodzicki} Wodzicki M. {\it Vestigia investiganda.} Mosc. Math. J. {\bf 2} (2002), no. 4, 769--798, 806.
\end{thebibliography}
\end{document}